\documentclass[12pt]{amsproc}
\usepackage[utf8]{inputenc}
\usepackage[T1]{fontenc}
\usepackage{amsmath,amsfonts,amssymb,amsthm,mathtools}
\usepackage{fullpage}
\usepackage{graphicx}
\usepackage{hyperref}
\hypersetup{
  colorlinks=true,
  citecolor=cyan,
  linkcolor=blue,
  urlcolor=cyan,
  pdfauthor={Davrbek Oltiboev; Furqat Vafoyev},
  pdftitle={Critical cusp cancellation and analytic response for a Bernoulli random interval map}
}

\newtheorem{theorem}{Theorem}[section]
\newtheorem{lemma}[theorem]{Lemma}
\newtheorem{proposition}[theorem]{Proposition}
\newtheorem{corollary}[theorem]{Corollary}
\theoremstyle{definition}

\theoremstyle{remark}

\numberwithin{equation}{section}
\newcommand{\bv}{\mathrm{BV}}
\newcommand{\ess}{\mathrm{ess}}

\title{Critical cusp cancellation and analytic response for a Bernoulli random interval map}

\author[D.~Oltiboev]{Davrbek Oltiboev}
\address{V.I. Romanovskiy Institute of Mathematics, Uzbekistan Academy of Sciences}
\address{Tashkent, Uzbekistan}
\email{davrbek.oltiboyev@gmail.com}

\author[F.~Vafoyev]{Furqat Vafoyev}
\address{National University of Uzbekistan named after M. Ulugbek}
\address{Tashkent, Uzbekistan}
\email{vafoyevfurqat@gmail.com}

\date{}

\begin{document}
\begin{abstract}
We show that the critical cusp exponent $\beta=-1/2$ admits $W^{2,1}$ invariant
 densities in an explicit family of expanding full-branch interval maps. The usual
 estimates exclude this exponent because two zero-order coefficients in the second
 derivative of the transfer operator are individually non-integrable. Their signed
 combination, however, contains the factor $\beta(2\beta+1)$ for a pure power cusp.
 We exploit its cancellation at $\beta=-1/2$ to prove a $W^{2,1}$ Lasota--Yorke
 inequality and a spectral gap for the explicit family. We also identify a separate
 obstruction caused by an interior cusp value. As an application, a small Bernoulli
 probability perturbation between two fixed maps has a stationary density that
 depends real-analytically on the perturbation parameter in $W^{2,1}$, with explicit
 formulas for all response derivatives.
\end{abstract}
\maketitle

\section{Introduction}

Invariant and stationary densities describe the statistical behaviour of a dynamical
system. For a density $h_\varepsilon$ and an observable $\varphi$, the stationary
expectation is $\int \varphi h_\varepsilon\,dx$. How does it change under a small
perturbation? Statistical stability asks for continuity, whereas linear response
asks for a first-order expansion. The latter allows perturbed averages to be
approximated from the unperturbed density and its response, and quantifies the
sensitivity of long-term statistics.

Uniform expansion alone does not ensure linear response for interval maps.
Perturbations of tent-like maps changing the turning-point value can destroy
differentiability when the one-sided slopes are finite \cite{BaladiSmania}.
Bahsoun and Galatolo \cite{BahsounGalatolo} showed that an unbounded cusp
derivative can instead restore linear response, because the vanishing
inverse-derivative weights improve the regularity of the transfer operator.
Oltiboev \cite{Oltiboev} obtained an $L^p$ version for $p>1$. Oltiboev,
Rakhimov and Ruziboev \cite{OltiboevRakhimovRuziboev} considered random
compositions with a common cusp point and value. For $-1<\beta<-1/2$, they
obtained $W^{2,1}$ stationary densities and response in $W^{1,1}$ under
perturbations of the probability law, following the viewpoint of
\cite{BahsounRuziboevSaussol}.

The present note is motivated by the strict upper endpoint in
\cite{OltiboevRakhimovRuziboev}. Its second-order estimate bounds two singular
zero-order terms separately by $t^{-2\beta-2}$, where $t$ is the distance to
the cusp. This weight is integrable for $\beta<-1/2$, but becomes $t^{-1}$ at
the endpoint. Does second-order regularity actually fail there, or do these
estimates discard an essential cancellation? The divergence of a sufficient
bound alone does not answer this question.

We answer it for an explicit full-branch family. The signed pure-power
coefficient contains $\beta(2\beta+1)$, which vanishes at $\beta=-1/2$.
Adding a constant to the critical power in the branch derivative leaves an
integrable remainder and yields a $W^{2,1}$ Lasota--Yorke inequality and a
spectral gap. The family is a critical specialization of the unperturbed model
of \cite{OltiboevRakhimovRuziboev}, also related to the model of
\cite{Oltiboev}. The contribution is the endpoint analysis, rather than an
unrelated construction.

An independent boundary issue remains. In \cite{OltiboevRakhimovRuziboev},
the first derivative of the transferred function has zero trace at the common
cusp value. At the endpoint that trace can be nonzero, even if the cusp value
is fixed. Its zero extension across an interior cusp value may then fail to
belong to $W^{2,1}$. We avoid this obstruction by placing the cusp value at~$1$.

Finally, we vary the Bernoulli probabilities of two fixed maps. The operator
is affine in the parameter and bounded on the same $W^{2,1}$ space, so the
spectral gap and standard perturbation theory \cite{Kato} give analytic
response. The stronger response topology reflects this restricted perturbation:
we do not extend the full deterministic or general random-law theorems of
\cite{Oltiboev,OltiboevRakhimovRuziboev} to the critical exponent.

\section{The critical coefficient}

Let $I=[0,1]$. Suppose that $f:I\to I$ has finitely many monotone branches,
each $C^1$ with nonzero derivative in its interior. Its transfer operator is
given, almost everywhere, by
\[
 (L_f\varphi)(x)=\sum_{y\in f^{-1}(x)}\frac{\varphi(y)}{|f'(y)|}.
\]
It is positive, preserves Lebesgue integrals, and is an $L^1(I)$ contraction.
An invariant probability density is a nonnegative fixed point with integral
one. The Sobolev spaces used below are taken on $(0,1)$, with no boundary
conditions, and have norms
$\|\varphi\|_{W^{k,1}}=\sum_{j=0}^{k}\|\varphi^{(j)}\|_1$,
for $k=1,2$.

When the branches are $C^3$ in their interiors, the zero-order coefficient in
the second derivative of $L_f\varphi$, away from endpoints of branch images,
is
\begin{equation}\label{eq:Df}
 D_f:=3\frac{(f'')^2}{(f')^4}-\frac{f'''}{(f')^3}.
\end{equation}
In the second-order estimate of \cite{OltiboevRakhimovRuziboev}, its absolute
value is controlled by
\[
 3\frac{|f''|^2}{|f'|^4}+\frac{|f'''|}{|f'|^3}
 =O(t^{-2\beta-2}).
\]
At the endpoint we must retain the signed combination in \eqref{eq:Df}
instead.

\begin{lemma}\label{lem:pure}
Let $c$ be the turning point and $t=|x-c|$. Suppose that the derivative
magnitude on one branch is $p(t)=Kt^\beta$, where $K>0$ and $\beta<0$.
Then, on either branch,
\begin{equation}\label{eq:pure}
 D_f=3\frac{(p')^2}{p^4}-\frac{p''}{p^3}
 =\frac{\beta(2\beta+1)}{K^2}\,t^{-2\beta-2}.
\end{equation}
In particular, this coefficient vanishes identically when $\beta=-1/2$.
\end{lemma}
\begin{proof}
On the increasing branch, $t=c-x$ and
$(f',f'',f''')=(p,-p',p'')$. On the decreasing branch,
$t=x-c$ and $(f',f'',f''')=(-p,-p',-p'')$.
Both substitutions give the first equality in \eqref{eq:pure}.
Since $p'=K\beta t^{\beta-1}$ and
$p''=K\beta(\beta-1)t^{\beta-2}$, the second equality follows from
$3\beta^2-\beta(\beta-1)=\beta(2\beta+1)$.
\end{proof}

\section{An explicit critical cusp family}

The unperturbed family in \cite[Section~4]{OltiboevRakhimovRuziboev} can be
written, with $c=1/2$, as
\[
 F_{\beta,\vartheta}(c\pm t)
 =1-\vartheta t-(1-\vartheta/2)(2t)^{\beta+1},
 \qquad 0\le t\le\frac12.
\]
Its specialization at $\beta=-1/2$ gives the square-root family below.
For $\vartheta=3/2$, it also agrees with the unperturbed formula in
\cite{Oltiboev} at root parameter $k=2$. These algebraic identities motivate
the example; neither earlier theorem covers this critical case.

The derivative magnitude is a constant plus $Kt^{-1/2}$. Integrating from a
cusp of height $1$ and imposing zero endpoint values fixes $K$.
Specifically, choose $1<\vartheta<2$ and set
\begin{equation}\label{eq:f}
 K_\vartheta:=\frac{1-\vartheta/2}{\sqrt2}>0,\qquad
 f_\vartheta(c\pm t):=1-\vartheta t-2K_\vartheta\sqrt t,
 \qquad 0\le t\le\frac12.
\end{equation}
Then $f_\vartheta(0)=f_\vartheta(1)=0$ and $f_\vartheta(c)=1$.
For $t=|x-c|>0$,
\begin{equation}\label{eq:exp}
 |f_\vartheta'(x)|=p_\vartheta(t):=\vartheta+K_\vartheta t^{-1/2},
 \qquad \inf_{x\ne c}|f_\vartheta'(x)|=1+\frac{\vartheta}{2}
 =:\Lambda_\vartheta>1.
\end{equation}
Thus $f_\vartheta$ has two full branches, is uniformly expanding, and has
cusp exponent $\beta=-1/2$. Write $\lambda_\vartheta:=\Lambda_\vartheta^{-1}<1$.

The inverse branches are also explicit. For $x\in I$, define
\begin{equation}\label{eq:u}
 u_\vartheta(x):=
 \frac{\sqrt{K_\vartheta^2+\vartheta(1-x)}-K_\vartheta}{\vartheta}.
\end{equation}
Then $g_{0,\vartheta}(x)=c-u_\vartheta(x)^2$ and
$g_{1,\vartheta}(x)=c+u_\vartheta(x)^2$. Consequently,
\begin{equation}\label{eq:explicitL}
 (L_\vartheta\varphi)(x)
 =\frac{u_\vartheta(x)}{K_\vartheta+\vartheta u_\vartheta(x)}
 \left[\varphi\bigl(c-u_\vartheta(x)^2\bigr)
 +\varphi\bigl(c+u_\vartheta(x)^2\bigr)\right].
\end{equation}
Here and below $L_\vartheta=L_{f_\vartheta}$.

\begin{lemma}\label{lem:critical}
For the family \eqref{eq:f}, on either branch with $t=|x-c|>0$,
\begin{equation}\label{eq:Dtheta}
 D_\vartheta(t):=3\frac{(f_\vartheta'')^2}{(f_\vartheta')^4}
 -\frac{f_\vartheta'''}{(f_\vartheta')^3}
 =-\frac{3K_\vartheta\vartheta}
 {4\sqrt t\,(K_\vartheta+\vartheta\sqrt t)^4},
\end{equation}
\begin{equation}\label{eq:lowerterms}
 \frac{|f_\vartheta''|}{|f_\vartheta'|^2}
 =\frac{K_\vartheta}{2\sqrt t\,(K_\vartheta+\vartheta\sqrt t)^2},
 \qquad
 \left|\frac{f_\vartheta''}{(f_\vartheta')^3}\right|
 =\frac{K_\vartheta}{2(K_\vartheta+\vartheta\sqrt t)^3}.
\end{equation}
The coefficients in \eqref{eq:Dtheta} and the first part of
\eqref{eq:lowerterms} belong to $L^1(0,1/2)$, and the last coefficient is
bounded.
\end{lemma}
\begin{proof}
From \eqref{eq:exp}, $p_\vartheta'=-(K_\vartheta/2)t^{-3/2}$ and
$p_\vartheta''=(3K_\vartheta/4)t^{-5/2}$.
Substituting these expressions and
$p_\vartheta=t^{-1/2}(K_\vartheta+\vartheta\sqrt t)$ into
$3(p_\vartheta')^2/p_\vartheta^4-p_\vartheta''/p_\vartheta^3$
gives \eqref{eq:Dtheta}. The same substitution gives
\eqref{eq:lowerterms}. Since $K_\vartheta>0$, the first two coefficients
are $O(t^{-1/2})$ and the last one is $O(1)$.
\end{proof}

\section{\texorpdfstring{$W^{2,1}$}{W2,1} regularity at the endpoint}

The estimate below follows the Lasota--Yorke strategy \cite{LasotaYorke} in a
Sobolev setting. The cancellation supplies the missing second-order
integrability. There is also a first-order change from
\cite{OltiboevRakhimovRuziboev}: at the critical exponent the coefficient
$|f_\vartheta''|/|f_\vartheta'|^2$ is in $L^1$ but not in $L^2$, so the
$(W^{1,1},L^2)$ estimate used there is not imported. We instead use
$W^{1,1}(I)\hookrightarrow L^\infty(I)$ to bound the lower-order terms, and
obtain quasi-compactness directly from the second-order estimate.

\begin{proposition}\label{prop:LY}
For each $\vartheta\in(1,2)$, the operator $L_\vartheta$ is bounded on
$W^{1,1}(I)$ and $W^{2,1}(I)$. There exists $C_\vartheta>0$ such that
\begin{equation}\label{eq:LY}
 \|L_\vartheta\varphi\|_{W^{2,1}}
 \le\lambda_\vartheta^2\|\varphi\|_{W^{2,1}}
 +C_\vartheta\|\varphi\|_{W^{1,1}},\qquad \varphi\in W^{2,1}(I).
\end{equation}
Moreover, the essential spectral radius of $L_\vartheta$ on $W^{2,1}(I)$
is at most $\lambda_\vartheta^2$.
\end{proposition}
\begin{proof}
For smooth $\varphi$, differentiation of \eqref{eq:explicitL}, or of the
inverse-branch formula, gives
\begin{equation}\label{eq:deriv1}
 (L_\vartheta\varphi)'
 =L_\vartheta\left(\frac{\varphi'}{f_\vartheta'}\right)
 -L_\vartheta\left(\frac{f_\vartheta''}{(f_\vartheta')^2}\varphi\right)
\end{equation}
and
\begin{equation}\label{eq:deriv2}
 \begin{aligned}
 (L_\vartheta\varphi)''
 &=L_\vartheta\left(\frac{\varphi''}{(f_\vartheta')^2}\right)
 -3L_\vartheta\left(\frac{f_\vartheta''}{(f_\vartheta')^3}\varphi'\right)\\
 &\quad+L_\vartheta\left(
 \left[3\frac{(f_\vartheta'')^2}{(f_\vartheta')^4}
 -\frac{f_\vartheta'''}{(f_\vartheta')^3}\right]\varphi\right).
 \end{aligned}
\end{equation}
Both branch images are $I$, so no characteristic function with an interior
endpoint occurs. The right-hand sides are integrable by Lemma~\ref{lem:critical};
hence these classical formulas on $(0,1)$ also give the weak derivatives
there. In particular, the two terms in the final coefficient of
\eqref{eq:deriv2} are kept together.

Using \eqref{eq:exp}, Lemma~\ref{lem:critical} and the $L^1$ contraction
property, we obtain
\begin{align*}
 \|(L_\vartheta\varphi)'\|_1
 &\le\lambda_\vartheta\|\varphi'\|_1+C\|\varphi\|_\infty
 \le C\|\varphi\|_{W^{1,1}},\\
 \|(L_\vartheta\varphi)''\|_1
 &\le\lambda_\vartheta^2\|\varphi''\|_1+C\|\varphi'\|_1 +2\|D_\vartheta\|_{L^1(0,1/2)}\|\varphi\|_\infty \le \lambda_\vartheta^2\|\varphi''\|_1+C\|\varphi\|_{W^{1,1}}.
\end{align*}
Together with $\|L_\vartheta\varphi\|_1\le\|\varphi\|_1$, these
estimates prove \eqref{eq:LY}. Smooth approximation establishes the stated
boundedness and inequality on the Sobolev spaces.

For the essential spectral radius, use the interpolation inequality
$\|\varphi\|_{W^{1,1}}\le\eta\|\varphi\|_{W^{2,1}}+C_\eta\|\varphi\|_1$,
valid for every $\eta>0$. For each $b\in(\lambda_\vartheta^2,1)$,
\eqref{eq:LY} therefore gives
$\|L_\vartheta\varphi\|_{W^{2,1}}
\le b\|\varphi\|_{W^{2,1}}+C_b\|\varphi\|_1$.
Iteration and $L^1$ contraction yield
\begin{equation}\label{eq:iterLY}
 \|L_\vartheta^n\varphi\|_{W^{2,1}}
 \le b^n\|\varphi\|_{W^{2,1}}
 +\frac{C_b}{1-b}\|\varphi\|_1.
\end{equation}
The compact embedding $W^{2,1}(I)\Subset L^1(I)$ and Hennion's criterion
\cite{Hennion} give $r_{\ess}\le b$.
Letting $b\downarrow\lambda_\vartheta^2$ proves the claim.
\end{proof}

\begin{proposition}\label{prop:gap}
For every $\vartheta\in(1,2)$, the map $f_\vartheta$ has a unique
absolutely continuous invariant probability measure. Its density
$h_\vartheta$ belongs to $W^{2,1}(I)$. On this space, $1$ is an
algebraically simple eigenvalue of $L_\vartheta$, and the rest of the
spectrum is contained in a disk of radius strictly smaller than one.
\end{proposition}
\begin{proof}
The weight $1/|f_\vartheta'|$, set equal to zero at the cusp, has bounded
variation: its absolute derivative on each branch is the first coefficient
in \eqref{eq:lowerterms}, which is integrable. Together with uniform
expansion, this gives the hypotheses of Rychlik's theory \cite{Rychlik}.
An invariant probability density exists, and the transfer operator admits
a finite cyclic Perron--Frobenius decomposition on $\bv(I)$; see also
\cite[Chapter~3]{Baladi}.

Each inverse branch has Lipschitz constant at most $\lambda_\vartheta$.
Thus a rank-$n$ cylinder has diameter at most $\lambda_\vartheta^n$,
and its closure is mapped onto $I$ by $f_\vartheta^n$.
Every nonempty open interval contains such a closed cylinder for
sufficiently large $n$. Consequently $f_\vartheta$ is topologically exact.

To see why the cyclic decomposition is trivial, take a nonzero nonnegative
$v\in\bv(I)$ with $L_\vartheta^qv=v$ for some $q\ge1$. There are an open
interval $U$ and $d>0$ such that $v\ge d$ almost everywhere on $U$.
Choose a cylinder $J\subset U$ of rank $n$ divisible by $q$, and let
$g_J:(0,1)\to\operatorname{int}J$ be its inverse branch. Its derivative
is nonzero on $(0,1)$, so
$v=L_\vartheta^nv\ge d|g_J'|>0$ almost everywhere there.
Thus two nonnegative densities fixed by the same power cannot have disjoint
supports. The finite cyclic decomposition reduces to a single mixing
component: the invariant probability density is unique in $\bv$ and $1$
is the only peripheral eigenvalue. The associated convergence extends
from $\bv$ to $L^1$ by density and $L^1$ contraction, so there are no
additional invariant probability densities in $L^1$.

Spectra are understood on the complexified spaces. On $W^{2,1}(I)$ let
$\ell(\varphi)=\int_I\varphi\,dx$.
Since $L_\vartheta^*\ell=\ell$, the number $1$ belongs to the spectrum.
Proposition~\ref{prop:LY} puts it outside the essential spectral disk, so
it is an eigenvalue. Every corresponding eigenvector lies in $\bv(I)$
and is a nonzero multiple of $h_\vartheta$.
This proves $h_\vartheta\in W^{2,1}(I)$ and geometric simplicity. The same
inclusion excludes all other eigenvalues of modulus one on $W^{2,1}(I)$.

A nontrivial Jordan block at $1$ would give $(L_\vartheta-I)v=h_\vartheta$
after rescaling. Then $L_\vartheta^nv=v+nh_\vartheta$, contradicting
$\|L_\vartheta^nv\|_1\le\|v\|_1$.
The $L^1$ contraction also excludes eigenvalues of modulus greater than
one. These observations, together with $r_{\ess}<1$, prove the asserted
spectral gap.
\end{proof}

\section{The obstruction caused by an interior cusp value}

In \cite{OltiboevRakhimovRuziboev}, the strict inequality $\beta<-1/2$ also
gives $|f''|/|f'|^3=O(t^{-2\beta-1})\to0$, which makes the first derivative
of the transferred function vanish at the common cusp value. At the critical
exponent this mechanism changes. The cancellation in \eqref{eq:Df} controls
the interior differential expression, but does not by itself remove a
nonzero derivative trace at an interior branch-image endpoint.

\begin{proposition}\label{prop:interior}
Suppose that a nonsingular map $f:I\to I$ has two monotone branches onto
$[0,a]$, where $0<a<1$, and is symmetric near its cusp $c$.
Assume that, for sufficiently small $t>0$,
\begin{equation}\label{eq:interior}
 f(c\pm t)=a-2K\sqrt t+R(t),\qquad K>0,
\end{equation}
where $R\in C^2((0,\delta))$ satisfies
\[
 R(t)=O(t),\qquad R'(t)=O(1),\qquad R''(t)=O(t^{-1}).
\]
For $\varphi\in C^1(I)$, let $L_f\varphi$ be extended by zero to $(a,1]$.
Then
\begin{equation}\label{eq:trace}
 (L_f\varphi)(x)=\frac{\varphi(c)}{K^2}(a-x)+o(a-x),
 \qquad \lim_{x\uparrow a}(L_f\varphi)'(x)
 =-\frac{\varphi(c)}{K^2}.
\end{equation}
If $\varphi(c)\ne0$, this zero extension does not belong to $W^{2,1}(I)$.
\end{proposition}
\begin{proof}
Write $p(t)=Kt^{-1/2}-R'(t)$ and
$q(t)=\varphi(c-t)+\varphi(c+t)$.
If $t=t(x)$ is the inverse distance to the cusp, then
\[
 a-x=2K\sqrt t+O(t),\qquad (L_f\varphi)(x)=\frac{q(t)}{p(t)}.
\]
As $q(t)\to2\varphi(c)$ and
$\sqrt t=(a-x)/(2K)+o(a-x)$,
the first assertion in \eqref{eq:trace} follows. Moreover,
$dt/dx=-1/p(t)$, so
\[
 (L_f\varphi)'(x)=-\frac{q'(t)}{p(t)^2}+\frac{q(t)p'(t)}{p(t)^3}.
\]
Here $q'$ is bounded,
$p'(t)=-(K/2)t^{-3/2}+O(t^{-1})$, and
$p(t)^3\sim K^3t^{-3/2}$. This proves the derivative limit in
\eqref{eq:trace}. The derivative of the zero extension is zero on
$(a,1)$. If $\varphi(c)\ne0$, its left and right traces at $a$ are
different, whereas every $W^{2,1}(I)$ function has a continuous first
derivative. This is a contradiction.
\end{proof}

For \eqref{eq:f}, the cusp value is $1$, so the nonzero derivative trace
causes no interior obstruction. The full-branch choice and the algebraic
cancellation therefore address distinct issues. Proposition~\ref{prop:interior}
does not rule out exceptional inputs with $\varphi(c)=0$.

\section{Bernoulli random perturbations and analytic response}

We now apply the probability-law viewpoint of
\cite{OltiboevRakhimovRuziboev,BahsounRuziboevSaussol} to an affine
perturbation of the critical family.
Fix distinct parameters $\vartheta_0,\vartheta_1\in(1,2)$, and write
$L_0=L_{\vartheta_0}$, $L_1=L_{\vartheta_1}$ and $D=L_1-L_0$.
Consider
\begin{equation}\label{eq:annealed}
 L_\varepsilon=(1-\varepsilon)L_0+\varepsilon L_1=L_0+\varepsilon D.
\end{equation}
For $0\le\varepsilon\le1$, this is the annealed transfer operator of the
system which independently chooses $f_{\vartheta_0}$ with probability
$1-\varepsilon$ and $f_{\vartheta_1}$ with probability $\varepsilon$ at
each step. Negative values of $\varepsilon$ are used only for analytic
continuation; they do not define Bernoulli probabilities. Put
\[
 W_0^{2,1}(I):=\left\{\varphi\in W^{2,1}(I):\int_I\varphi\,dx=0\right\}.
\]

\begin{theorem}\label{thm:analytic}
There exists $\varepsilon_0\in(0,1)$ such that, for every real
$|\varepsilon|<\varepsilon_0$, the operator $L_\varepsilon$ has an
isolated algebraically simple eigenvalue $1$ on $W^{2,1}(I)$ and a unique
normalized fixed vector $h_\varepsilon\in W^{2,1}(I)$, with
$\int_I h_\varepsilon\,dx=1$.
The map $\varepsilon\mapsto h_\varepsilon\in W^{2,1}(I)$ is real analytic.
For $0\le\varepsilon<\varepsilon_0$, this vector is nonnegative and hence
a stationary density of the Bernoulli system.

Let $h_0=h_{\vartheta_0}$ and
\[
 R_0=(I-L_0)^{-1}:W_0^{2,1}(I)\longrightarrow W_0^{2,1}(I).
\]
Then
\begin{equation}\label{eq:LR}
 h_0'=R_0Dh_0=(I-L_0)^{-1}(L_1-L_0)h_0,
\end{equation}
and, for every $n\ge1$,
\begin{equation}\label{eq:higher}
 \left.\frac{d^n}{d\varepsilon^n}h_\varepsilon\right|_{\varepsilon=0}
 =n!(R_0D)^n h_0.
\end{equation}
\end{theorem}
\begin{proof}
By Proposition~\ref{prop:gap}, $1$ is an isolated algebraically simple
eigenvalue of $L_0$ on $W^{2,1}(I)$, with the remaining spectrum strictly
inside the unit disk. Since $D$ is bounded on this space, \eqref{eq:annealed}
is analytic in operator norm. Analytic perturbation theory \cite{Kato}
gives an analytic rank-one spectral projection and an eigenvector
$v_\varepsilon$ near $h_0$. Its integral is nonzero for small
$\varepsilon$. Integrating its eigenvalue equation and using integral
preservation shows that the associated eigenvalue is exactly $1$.
Normalization yields the analytic family $h_\varepsilon$.

After decreasing $\varepsilon_0$, spectral separation keeps the remaining
spectrum inside a disk of radius less than one. The spectral projection is
$P_\varepsilon\varphi=h_\varepsilon\int_I\varphi\,dx$, and therefore
\begin{equation}\label{eq:limitdensity}
 L_\varepsilon^n1\longrightarrow h_\varepsilon
 \quad\text{in }W^{2,1}(I),\qquad n\to\infty.
\end{equation}
For $0\le\varepsilon<\varepsilon_0$, each function on the left is
nonnegative and has integral one. Thus the limit is a stationary probability
density.

The spectral projection of $L_0$ is $P_0\varphi=h_0\int_I\varphi\,dx$,
so the restriction of $L_0$ to $W_0^{2,1}(I)=\ker P_0$ has spectral radius
less than one. Hence $R_0=\sum_{k=0}^\infty L_0^k$ converges in operator
norm on $W_0^{2,1}(I)$. Also, $D$ maps $W^{2,1}(I)$ into this
zero-integral subspace, because both $L_0$ and $L_1$ preserve integrals.
Thus $R_0D$ is well defined on $W^{2,1}(I)$.

Differentiating $L_\varepsilon h_\varepsilon=h_\varepsilon$ and
$\int_Ih_\varepsilon\,dx=1$ at zero gives
$(I-L_0)h_0'=Dh_0$ and $\int_Ih_0'\,dx=0$, which proves \eqref{eq:LR}.
Since \eqref{eq:annealed} is affine, repeated differentiation yields
\[
 (I-L_0)h_0^{(n)}=nDh_0^{(n-1)},\qquad
 \int_Ih_0^{(n)}\,dx=0,\qquad n\ge1.
\]
Induction gives \eqref{eq:higher}. Equivalently, for sufficiently small
$|\varepsilon|$,
\begin{equation}\label{eq:powerSeries}
 h_\varepsilon=h_0+\sum_{n=1}^\infty\varepsilon^n(R_0D)^nh_0
 \quad\text{in }W^{2,1}(I).
\end{equation}
\end{proof}

\begin{corollary}\label{cor:observables}
For every bounded linear functional $\Phi$ on $W^{2,1}(I)$,
$\varepsilon\mapsto\Phi(h_\varepsilon)$ is analytic near zero, and
\[
 \left.\frac{d^n}{d\varepsilon^n}\Phi(h_\varepsilon)\right|_{\varepsilon=0}
 =n!\,\Phi\bigl((R_0D)^nh_0\bigr),\qquad n\ge1.
\]
In particular, for every $\varphi\in L^\infty(I)$,
\begin{equation}\label{eq:observables}
 \left.\frac{d}{d\varepsilon}\int_I\varphi h_\varepsilon\,dx\right|_{\varepsilon=0}
 =\int_I\varphi R_0(L_1-L_0)h_0\,dx.
\end{equation}
\end{corollary}

\section{Concluding remarks}

The example shows that the strict restriction $\beta<-1/2$ in the separate
Sobolev coefficient estimates is not an intrinsic barrier for every cusp map.
At the endpoint, the factor $2\beta+1$ cancels the nominally non-integrable
leading term, while the full-branch geometry avoids an interior jump in the
first derivative. Together these properties give $W^{2,1}$ invariant densities
and analytic response to a Bernoulli probability perturbation.

A natural continuation is to replace the explicit derivative
$\vartheta+K_\vartheta t^{-1/2}$ by an expansion
$Kt^{-1/2}+A+O(t^{1/2})$ with compatible estimates on its first two
derivatives. The resulting general endpoint theory, as well as deterministic
perturbations of the maps themselves, is not pursued here.

\subsection*{Acknowledgements}
This research was supported by the Ministry of Higher Education, Science and
Innovation of the Republic of Uzbekistan under State Grant No.\ FL-9524115114.

\end{document}